\documentclass[12pt,reqno]{amsart}%
\usepackage{amssymb}
\usepackage{amsmath}
\usepackage{amsfonts}
\usepackage{graphicx}%
\providecommand{\U}[1]{\protect\rule{.1in}{.1in}}
\providecommand{\U}[1]{\protect\rule{.1in}{.1in}}
\providecommand{\U}[1]{\protect\rule{.1in}{.1in}}
\newtheorem{theorem}{Theorem}
\theoremstyle{plain}

\newtheorem{corollary}{Corollary}

\newtheorem{lemma}{Lemma}

\newtheorem{proposition}{Proposition}

\numberwithin{equation}{section}
\begin{document}
\title[Global Solutions and Construction of Invariant Regions]{Existence of Global Solutions Via Invariant Regions for a Generalized
Reaction-Diffusion System with a Tri-diagonal Toeplitz Matrix of Diffusion Coefficients}
\author{Salem ABDELMALEK}
\address{S. ABDELMALEK: Department of Mathematics, College of Sciences, Yanbu, Taibah
University, Saudi Arabia. \newline and,\newline
 Department of
mathematics, University of Tebessa 12002 Algeria.}
\email{sabdelmalek@taibahu.edu.sa} \subjclass[2000]{Primary 35K45,
35K57} \keywords{Reaction-Diffusion Systems, Invariant Regions,
Global Existence}

\begin{abstract}
The aim of this paper is to construct invariant regions of a generalized
$m$-component reaction-diffusion system with a tri-diagonal Toeplitz matrix of
diffusion coefficients and prove the global existence of solutions using
Lyapunov functional. The paper assumes nonhomogeneous boundary conditions and
polynomial growth for the non-linear reaction term.

\end{abstract}
\maketitle

\section{\textbf{Introduction}}

Reaction-diffusion systems arise in many applications ranging from
chemistry and biology to engineering. They have been the subject of
countless studies in the past few decades. One of the most important
aspects of this broad field is proving the global existence of
solutions under certain assumptions and restrictions. We quote the
recent papers of Amann \cite{Amann1, Amann2} who studies the problem
in $W^{1,p}$ and $W^{2,p}$ spaces with $p>n.$ An excellent reference
for a dynamic theory of reaction-diffusion systems is the book of
Henry \cite{Henry}.

In 2001, Kouachi \cite{Kouachi2} followed on previous work and showed the
global existence of solutions assuming the reaction terms of a $2\times2$
diagonal system exhibit a polynomial growth. This was later generalized by
Kouachi for an arbitrary $2\times2$ Toeplitz matrix. In \cite{Abdelmalek1},
the author of this work studied the $3\times3$ case under the same assumptions
and restrictions. Abdelmalek and Kouachi \cite{Abdelmalek3} also showed the
global existence of solutions for an $m$-component reaction--diffusion system
($m\geq2$) with a diagonal diffusion matrix and reaction terms of polynomial growth.

An important factor in the study of reaction diffusion systems is
the characteristics of the diffusion matrix. Although in some cases
the matrix is diagonal, in many cases cross diffusion terms exist.
For instance, many chemical and biological operations are described
by reaction-diffusion systems with a tri-diagonal matrix of
diffusion coefficients, (see, e.g., Cussler \cite{Cussler1} and
\cite{Cussler2}). Other examples include the modelling of epidemics
\cite{Capasso}, ecology \cite{Okubo} and biochemistry \cite{Aris},
where cross-diffusion appears to be a very relevant problem to be
analyzed. In this paper, tri-diagonal diffusion matrices have been
considered and sufficient conditions have been given for global
existence steady states.

The purpose of this paper is to prove the global existence of solutions with
nonhomogeneous Neumann, Dirichlet, or Robin conditions and a polynomial growth
of reaction terms. The polynomial growth is established through a mere single
inequality as we shall show. The main contribution of this paper is the fact
that we take a general Toeplitz matrix as opposed to the symmetry constraint
assumed in \cite{Abdelmalek0}.

Throughout this paper, we consider an $m$-component system, with $m\geq2$:%
\begin{equation}
\dfrac{\partial U}{\partial t}-D\Delta U=F\left(  U\right)  \text{\ in }%
\Omega\times\left(  0,+\infty\right)  , \label{1.1}%
\end{equation}
with the boundary conditions:%
\begin{equation}
\alpha U+\left(  1-\alpha\right)  \partial_{\eta}U=B\text{\ on }\partial
\Omega\times\left(  0,+\infty\right)  , \label{1.2}%
\end{equation}
or%
\begin{equation}
\alpha U+\left(  1-\alpha\right)  D\partial_{\eta}U=B\text{\ on }%
\partial\Omega\times\left(  0,+\infty\right)  , \label{1.02}%
\end{equation}
in the case of non-diagonal boundary conditions, and the initial data:%
\begin{equation}
U\left(  x,0\right)  =U_{0}\left(  x\right)  \text{ on}\;\Omega. \label{1.3}%
\end{equation}
We consider three types of boundary conditions:

\begin{enumerate}
\item[(i)] Nonhomogeneous Robin boundary conditions, corresponding to%
\[
0<\alpha<1,\text{ }B\in%
\mathbb{R}
^{m};
\]

\item[(ii)] Homogeneous Neumann boundary conditions, corresponding to%
\[
\alpha=0\text{ and }B\equiv0;
\]

\item[(iii)] Homogeneous Dirichlet boundary conditions, corresponding to%
\[
1-\alpha=0\text{ and }B\equiv0.
\]

\end{enumerate}

In the context of this work, $\Omega$ is an open bounded domain of class
$C^{1}$ in $\mathbb{%
\mathbb{R}
}^{n}$ with boundary $\partial\Omega$, $\dfrac{\partial}{\partial\eta}$
denotes the outward normal derivative on $\partial\Omega$, and\
\begin{align*}
U  &  :=\left(  u_{1},...,u_{m}\right)  ^{T},\\
F  &  :=\left(  f_{1},...,f_{m}\right)  ^{T},\\
B  &  :=\left(  \beta_{1},...,\beta_{m}\right)  ^{T}.
\end{align*}
The diffusion matrix is assumed to be a tri-diagonal Toeplitz one of the form%
\[
D:=\left(
\begin{array}
[c]{ccccc}%
a & b & 0 & \cdots & 0\\
c & a & b & \ddots & \vdots\\
0 & c & a & \ddots & 0\\
\vdots & \ddots & \ddots & \ddots & b\\
0 & \cdots & 0 & c & a
\end{array}
\right)  _{m\times m},
\]
where $a,b$ and $c$ are supposed to be strictly positive constants
satisfying:
\begin{equation}
\cos\frac{\pi}{m+1}<\frac{a}{b+c}, \label{1.4}%
\end{equation}
which reflects the parabolicity of the system.

The initial data are assumed to be in the regions:%
\begin{equation}
\Sigma_{\mathfrak{L},\mathfrak{Z}}:=\left\{  U_{0}\in\mathbb{%
\mathbb{R}
}^{m}:\left\langle V_{z},U_{0}\right\rangle \leq0\leq\left\langle V_{\ell
},U_{0}\right\rangle ,\text{ }\ell\in\mathfrak{L},\text{ }z\in\mathfrak{Z}%
\right\}  , \label{1.5}%
\end{equation}%
\begin{equation}
\mathfrak{L}\cap\mathfrak{Z}=\varnothing,\mathfrak{L}\cup\mathfrak{Z}=\left\{
1,2,...,m\right\}  , \label{1.5-a}%
\end{equation}
subject to
\[
\left\langle V_{z},B\right\rangle \leq0\leq\left\langle V_{\ell}%
,B\right\rangle ,\text{ }\ell\in\mathfrak{L},z\in\mathfrak{Z}.
\]
The vector $V_{\ell}=\left(  v_{1\ell},...,v_{m\ell}\right)  ^{T}$ are defined
as%
\[
v_{k\ell}=\sqrt{\mu^{k}}\sin\frac{k\left(  m+1-\ell\right)  \pi}{m+1},\text{
}k=1,...,m,
\]
with%
\[
\mu:=\frac{b}{c}.
\]
The notation $\left\langle \cdot,\cdot\right\rangle $ denotes the inner
product in $%
\mathbb{R}
^{m}.$

From (\ref{1.5-a}) we can clearly see that there are in fact $2^{m}$ regions.
One of the main contributions of this paper is that unlike previous studies we
cover \emph{all} possible regions. Hence, the work carried out here is a
generalization of previous studies. The most important of these studies are
discussed below.

In 2002, Kouachi \cite{Kouachi1} studied the case $m=2$, for which
the parabolicity condition we use here (\ref{1.4}) reduces to the
same condition employed in \cite{Kouachi1}: $2a>\left(  b+c\right)
$. Although in this case $2^{2}=4$ regions exist, the study of
Kouachi considered only a couple of these regions. Setting $m=2$ in
(\ref{1.5}) yields the following regions:

\begin{itemize}
\item If $\mathfrak{L}=\left\{  1,2\right\}  ,\mathfrak{Z}=\emptyset$ then,%
\[
\Sigma_{\mathfrak{L},\mathfrak{Z}}=\left\{  \left(  u_{1}^{0},u_{2}%
^{0}\right)  ^{T}\in\mathbb{%
\mathbb{R}
}^{2}:u_{1}^{0}\geq\sqrt{\mu}\left\vert u_{2}^{0}\right\vert \text{ if }%
\beta_{1}\geq\sqrt{\mu}\left\vert \beta_{2}\right\vert \right\}  .
\]

\item If $\mathfrak{L}=\left\{  2\right\}  ,\mathfrak{Z}=\left\{  1\right\}  $
then,
\[
\Sigma_{\mathfrak{L},\mathfrak{Z}}=\left\{  \left(  u_{1}^{0},u_{2}%
^{0}\right)  ^{T}\in\mathbb{%
\mathbb{R}
}^{2}:\sqrt{\mu}u_{2}^{0}\geq\left\vert u_{1}^{0}\right\vert \text{ if }%
\sqrt{\mu}\beta_{2}\geq\left\vert \beta_{1}\right\vert \right\}  .
\]

\item If $\mathfrak{L}=\emptyset,\mathfrak{Z}=\left\{  1,2\right\}  $ then,%
\[
\Sigma_{\mathfrak{L},\mathfrak{Z}}=\left\{  \left(  u_{1}^{0},u_{2}%
^{0}\right)  ^{T}\in\mathbb{%
\mathbb{R}
}^{2}:-u_{1}^{0}\geq\sqrt{\mu}\left\vert u_{2}^{0}\right\vert \text{ if
}-\beta_{1}\geq\sqrt{\mu}\left\vert \beta_{2}\right\vert \right\}  .
\]

\item If $\mathfrak{L}=\left\{  1\right\}  ,\mathfrak{Z}=\left\{  2\right\}  $
then,%
\[
\Sigma_{\mathfrak{L},\mathfrak{Z}}=\left\{  \left(  u_{1}^{0},u_{2}%
^{0}\right)  ^{T}\in\mathbb{%
\mathbb{R}
}^{2}:-\sqrt{\mu}u_{2}^{0}\geq\left\vert u_{1}^{0}\right\vert \text{ if
}-\sqrt{\mu}\beta_{2}\geq\left\vert \beta_{1}\right\vert \right\}  .
\]

\end{itemize}

In fact the last two of these regions were not considered in \cite{Kouachi1}.

In 2007, the author of this work \cite{Abdelmalek1} studied the case $m=3$ for
which the parabolicity condition is $\sqrt{2}a>\left(  b+c\right)  $,
resulting from the direct substitution of $m=3$ in (\ref{1.4}). The total
number of regions in this case is $2^{3}=8$ of which only 4 regions were
however studied.

In 2014 the author \cite{Abdelmalek0} elaborated on the generalized
$m$-component case with a tri-diagonal matrix having equal upper and lower
diagonal elements, i.e. $\left(  b=c\right)  $. Substituting $b=c$ in
(\ref{1.4}) yields the same condition used in \cite{Abdelmalek0}: $2b\cos
\frac{\pi}{m+1}<a$.

The aim of this work is to prove the global existence of solutions. The
necessary proofs are similar for all the invariant regions. Hence we only
focus on one of the regions and present a generalization at the end of the paper.

Consider the region with $\mathfrak{L}=\left\{  1,2,...,m\right\}  $ and
$\mathfrak{Z}=\emptyset$ yielding%
\begin{equation}
\Sigma_{\mathfrak{L},\emptyset}=\left\{  U_{0}\in\mathbb{%
\mathbb{R}
}^{m}:\left\langle V_{\ell},U_{0}\right\rangle \geq0,\text{ }\ell
\in\mathfrak{L},\right\}  , \label{1.6}%
\end{equation}
subject to%
\[
\left\langle V_{\ell},B\right\rangle \geq0,\text{ }\ell\in\mathfrak{L.}%
\]

In order to establish the global existence of solutions in this region we
diagonalize the diffusion matrix $D$. We define the reaction diffusion
functions as:%
\begin{equation}
\digamma\left(  W\right)  :=\left(  \digamma_{1},\digamma_{2},...,\digamma
_{m}\right)  ^{T},\text{ }\digamma_{\ell}:=\left\langle V_{\ell}%
,F\right\rangle , \label{1.8}%
\end{equation}
where the variable $W=\left(  w_{1},w_{2},...,w_{m}\right)  ^{T}$ is given by%
\begin{equation}
W:=\left(  w_{1},w_{2},...,w_{m}\right)  ^{T},w_{\ell}:=\left\langle V_{\ell
},U\right\rangle . \label{1.7}%
\end{equation}
The functions $\digamma_{\ell}$ must satisfy the following three conditions:

\begin{enumerate}
\item[(A1)] be continuously differentiable on $%
\mathbb{R}
_{+}^{m}$ for all $\ell=1,...,m$, satisfying $\digamma_{\ell}(w_{1}%
,...,w_{\ell-1},0,w_{\ell+1},...,w_{m})\geq0$, for all $w_{\ell}\geq0;$
$\ell=1,...,m $.

\item[(A2)] be of polynomial growth (see the work of Hollis and Morgan
\cite{Hollis3}), which means that for all $\ell=1,...,m$:%
\begin{equation}
\left\vert \digamma_{\ell}\left(  W\right)  \right\vert \leq C_{1}\left(
1+\left\langle W,1\right\rangle \right)  ^{N},n\in%
\mathbb{N}
\text{,on }\left(  0,+\infty\right)  ^{m}. \label{1.10}%
\end{equation}

\item[(A3)] satisfy the inequality:%
\begin{equation}
\left\langle S,\digamma\left(  W\right)  \right\rangle \leq C_{2}\left(
1+\left\langle W,1\right\rangle \right)  ,\label{1.11}%
\end{equation}
where%
\[
S:=\left(  d_{1},d_{2},...,d_{n-1},1\right)  ^{T},
\]
for all $w_{\ell}\geq0,$ $\ell=1,...,m,$. All the constants $d_{\ell}$ satisfy
$d_{\ell}\geq\overline{d_{\ell}},$ $\ell=1,...,m$ where $\overline{d_{\ell}},$
$\ell=1,...,m,$ are sufficiently large positive constants. Here $C_{1}$ and
$C_{2}$ are uniformly bounded positive functions defined on $\mathbb{R}%
_{+}^{m}$.
\end{enumerate}

\section{Some Properties of the Diffusion Matrix and Parabolicity}

\label{sec:prelim}

\begin{proposition}
A quadratic form $Q=\left\langle X,AX\right\rangle =X^{T}AX$, with $A$ being a
symmetric matrix, is positive definite for every non-zero column vector $X$ if
all the principal minors in the top-left corner of $A$ are positive. If $A$ is
non-symmetric, $Q$\ is positive definite iff the principal minors in the
top-left corner of $\frac{1}{2}\left(  A+A^{T}\right)  $ are positive.
\end{proposition}

\begin{lemma}
The reaction-diffusion system (\ref{1.1}) satisfies\ the parabolicity
condition if (\ref{1.4}) is satisfied.
\end{lemma}

\begin{proof}
The system (\ref{1.1}) satisfies the parabolicity condition if the matrix
$\left(  D+D^{T}\right)  $ is positive definite. The matrix $\left(
D+D^{T}\right)  $ is symmetric tri-diagonal with off-diagonal elements
$\frac{1}{2}\left(  b+c\right)  $. In \cite{Abdelmalek0} a similar matrix with
off-diagonal elements $b$ and the parabolicity condition
\[
2b\cos\frac{\pi}{m+1}<a,
\]
is considered. Substituting $b$ with $\frac{1}{2}\left(  b+c\right)  $ yields
(\ref{1.4}).
\end{proof}

\begin{lemma}
[see \cite{Carl} ]\label{Lemma0}The eigenvalues $\overline{\lambda}_{\ell
}<\overline{\lambda}_{\ell-1};$ $\ell=2,...,m,$ of $D^{T}$ are positive
and\ are given by
\begin{equation}
\overline{\lambda}_{\ell}:=a+2\sqrt{bc}\cos\left(  \frac{\ell\pi}{m+1}\right)
, \label{2.2}%
\end{equation}
with the corresponding eigenvectors being\ $\overline{V}_{\ell}=V_{m+1-\ell},$
for $\ell=1,...,m$. Therefore, $D^{T}$ is diagonalizable.
\end{lemma}

In the remainder of this work we require an ascending order of the
eigenvalues. In order to simplify the indices in the formulas to come we
define%
\begin{equation}
\lambda_{\ell}:=\bar{\lambda}_{m+1-\ell}=a+2\sqrt{bc}\cos\left(
\frac{(m+1-\ell)\pi}{m+1}\right)  ;\text{ }\ell=1,...,m, \label{2.3}%
\end{equation}
thus $\lambda_{\ell}<\lambda_{\ell+1};$ $\ell=2,...,m.$

\begin{proof}
Recall that the diffusion matrix is positive definite, implying that its
eigenvalues are necessarily positive. For a given eigenpair $\left(
\overline{\lambda},X\right)  $ the components of $\left(  D^{T}-\overline
{\lambda}I\right)  X=0$ are%
\[
bx_{k-1}+\left(  a-\overline{\lambda}\right)  x_{k}+cx_{k+1}=0,\text{
}k=1,...,m,
\]
with $x_{0}=x_{m+1}=0,$ or equivalently,
\[
x_{k+2}+\left(  \frac{a-\overline{\lambda}}{c}\right)  x_{k+1}+\mu
x_{k}=0,k=0,...,m-1,
\]
whose solutions are%
\[
x_{k}=\left\{
\begin{array}
[c]{ll}%
\alpha r_{1}^{k}+\beta r_{2}^{k}, & \text{if }r_{1}\neq r_{2},\\
\alpha\rho^{k}+\beta k\rho^{k}, & \text{if }r_{1}=r_{2}=\rho,
\end{array}
\right.
\]
where $\alpha$ and $\beta$ are arbitrary constants.

For the eigenvalue problem at hand, $r_{1}$ and $r_{2}$ must be distinct.
Putting $x_{k}=\alpha r_{1}^{k}+\beta r_{2}^{k}$, and $x_{0}=x_{m+1}=0$
yields
\[
\left\{
\begin{array}
[c]{l}%
0=\alpha+\beta\\
0=\alpha r_{1}^{m+1}+\beta r_{2}^{m+1}%
\end{array}
\right.  \Rightarrow\left(  \frac{r_{1}}{r_{2}}\right)  ^{m+1}=\frac{-\beta
}{\alpha}=1\Rightarrow\frac{r_{1}}{r_{2}}=e^{\frac{2i\pi\ell}{m+1}}.
\]
Therefore we see that $r_{1}=r_{2}e^{\frac{2i\pi\ell}{m+1}}$ for $1\leq
\ell\leq m$. This together with
\[
r^{2}+\left(  \frac{a-\overline{\lambda}}{c}\right)  r+\mu=\left(
r-r_{1}\right)  \left(  r-r_{2}\right)  \Rightarrow\left\{
\begin{array}
[c]{l}%
r_{1}r_{2}=\mu\\
r_{1}+r_{2}=-\frac{a-\overline{\lambda}}{c}%
\end{array}
\right.  ,
\]
leads to $r_{1}=\sqrt{\mu}e^{\frac{i\pi\ell}{m+1}}$, $r_{2}=\sqrt{\mu
}e^{-\frac{i\pi\ell}{m+1}}$, and
\[
\overline{\lambda}=a+2\sqrt{cb}\left(  e^{\frac{i\pi\ell}{m+1}}+e^{-\frac
{i\pi\ell}{m+1}}\right)  =a+2a+2\sqrt{cb}\cos\left(  \frac{\ell\pi}%
{m+1}\right)  \text{.}%
\]
Thus the eigenvalues of $D^{T}$ are given by
\[
\overline{\lambda}_{\ell}=a+2\sqrt{cb}\cos\left(  \frac{\ell\pi}{m+1}\right)
,
\]
for $\ell=1,...,m.$

Since the eigenvalues are all distinct (as $\cos\theta$ is strictly decreasing
on $\left(  0,\pi\right)  ,$ and $b\neq0\neq c$), then $D$ is necessarily diagonalizable.

The $\ell^{th}$ component of any eigenvector associated with $\lambda_{\ell}$
satisfies $x_{k}=\alpha r_{1}^{k}+\beta r_{2}^{k}$, with $\alpha+\beta=0$.
Thus
\[
x_{k}=\alpha\mu^{\frac{k}{2}}\left(  e^{\frac{2i\pi k}{m+1}}-e^{-\frac{2i\pi
k}{m+1}}\right)  =2i\alpha\mu^{\frac{k}{2}}\sin\left(  \frac{k}{m+1}%
\pi\right)  .
\]
Setting $\alpha=\frac{1}{2i}$ yields a particular eigenvector associated to
$\overline{\lambda}_{\ell}$ given by
\[
\overline{V}_{\ell}=\left(  \mu^{\frac{1}{2}}\sin\left(  \frac{1\ell\pi}%
{m+1}\right)  ,\mu^{\frac{2}{2}}\sin\left(  \frac{2\ell\pi}{m+1}\right)
,...,\mu^{\frac{m}{2}}\sin\left(  \frac{m\ell\pi}{m+1}\right)  \right)  ^{t}.
\]
Since the eigenvectors are all distinct then $\left\{  \overline{V}%
_{1},\overline{V}_{2},...,\overline{V}_{m}\right\}  $ is a complete linearly
independent set, hence $\left(  \overline{V}_{1}\shortmid\overline{V}%
_{2}\shortmid...\shortmid\overline{V}_{m}\right)  $ \emph{diagonalizes} $D$.

Now let us prove that
\[
\overline{\lambda}_{\ell}<\overline{\lambda}_{\ell-1};\text{ }\ell=2,...,m.
\]
We have
\[
\ell>\ell-1\Rightarrow\frac{\ell\pi}{m+1}>\frac{\left(  \ell-1\right)  \pi
}{m+1}.
\]
Once again using the fact that $\cos\theta$ is strictly decreasing on $\left(
0,\pi\right)  ,$ we deduce that
\[
\cos\left(  \frac{\ell\pi}{m+1}\right)  <\cos\left(  \frac{\left(
\ell-1\right)  \pi}{m+1}\right)  ,
\]
whereupon%
\[
\overline{\lambda}_{\ell}=a+2\sqrt{cb}\cos\left(  \frac{\ell\pi}{m+1}\right)
<a+2\sqrt{cb}\cos\left(  \frac{\left(  \ell-1\right)  \pi}{m+1}\right)
=\overline{\lambda}_{\ell-1}.
\]

\end{proof}

\begin{lemma}
The eigenvalues of the matrix $D$ are positive, i.e. $\lambda_{\ell}>0$ and
$\det D>0.$
\end{lemma}

\begin{proof}
Recall that $\lambda_{\ell}<\lambda_{\ell+1};$ $\ell=1,...,m-1$, i.e.%
\[
\lambda_{1}<\lambda_{2}<...<\lambda_{m}.
\]
We want to show that $\lambda_{1}>0$. First, we have
\begin{equation}
\lambda_{1}=a+2\sqrt{cb}\cos\left(  \frac{m}{m+1}\pi\right)  >0, \label{**1}%
\end{equation}
which implies%
\[
a>2\sqrt{bc}\left[  -\cos\left(  \frac{m}{m+1}\pi\right)  \right]  .
\]
From condition (\ref{1.4}), we obtain%
\begin{equation}
a>\left(  c+b\right)  \left(  \cos\frac{\pi}{m+1}\right)  . \label{**2}%
\end{equation}
Note that
\begin{equation}
\frac{m}{m+1}>\frac{1}{2}\Rightarrow\cos\left(  \frac{m}{m+1}\pi\right)
<\cos\frac{\pi}{2}=0; \label{**3}%
\end{equation}
furthermore
\begin{equation}
\left(  c+b\right)  \geq2\sqrt{bc}. \label{**4}%
\end{equation}
We also have%
\begin{equation}
\cos\left(  \frac{\pi}{m+1}\right)  +\cos\left(  \frac{m}{m+1}\pi\right)  =0
\label{**5}%
\end{equation}
since%
\begin{align*}
\cos\left(  \frac{\pi}{m+1}\right)  +\cos\left(  \frac{m}{m+1}\pi\right)   &
=2\cos\left(  \frac{\frac{\pi}{m+1}+\frac{m}{m+1}\pi}{2}\right)  \cos\left(
\frac{\frac{\pi}{m+1}-\frac{m}{m+1}\pi}{2}\right) \\
&  =2\cos\left(  \frac{\pi}{2}\right)  \cos\left(  \frac{m-1}{m+1}\frac{\pi
}{2}\right)  =0.
\end{align*}

Now, from (\ref{**3}), (\ref{**4}), and (\ref{**5}), we obtain%
\begin{equation}
\left(  c+b\right)  \left(  \cos\frac{\pi}{m+1}\right)  \geq2\sqrt{bc}\left[
-\cos\left(  \frac{m}{m+1}\pi\right)  \right]  , \label{**6}%
\end{equation}
and from (\ref{**2}) and (\ref{**6}), we get%
\[
a>2\sqrt{bc}\left[  -\cos\left(  \frac{m}{m+1}\pi\right)  \right]  ,
\]
which concludes the proof of (\ref{**1}) and guarantees that all eigenvalues
of $D^{T}$ are positive. Furthermore since the eigenvalues of $D$ are the same
as those of $D^{T}$ we conclude that $\det D>0$.
\end{proof}

\section{Local Existence and Invariant Regions}

The usual norms in spaces $L^{p}(\Omega)$, $L^{\infty}(\Omega)$ and
$C(\overline{\Omega})$ are denoted respectively by:%
\begin{align}
\left\Vert u\right\Vert _{p}^{p}  &  =\frac{1}{\left\vert \Omega\right\vert
}\int_{\Omega}\left\vert u(x)\right\vert ^{p}dx;\nonumber\\
\left\Vert u\right\Vert _{\infty}  &  =ess\underset{x\in\Omega}{\sup
}\left\vert u(x)\right\vert ,\nonumber
\end{align}
and%
\[
\left\Vert u\right\Vert _{C(\overline{\Omega})}=\underset{x\in\overline
{\Omega}}{\max}\left\vert u(x)\right\vert .
\]

It is well-known that in order to prove the global existence of solutions to a
reaction-diffusion system (see Henry \cite{Henry}) it suffices to derive a
uniform estimate of the associated reaction term on $\left[  0,T_{\max
}\right)  $ in the space $L^{p}(\Omega)$ for some $p>n/2$. Our aim is to
construct Lyapunov polynomial functionals allowing us to obtain $L^{p}$-bounds
on the components, which leads to global existence. Since the reaction terms
are continuously differentiable on $%
\mathbb{R}
_{+}^{m}$, then for any initial data in $C(\overline{\Omega})$ it is
straightforward to directly check their Lipschitz continuity on bounded
subsets of the domain of a fractional power of the operator
\begin{equation}
\mathfrak{D}=-\operatorname*{diag}\left(  \lambda_{1}\Delta,\lambda_{2}%
\Delta,...,\lambda_{m}\Delta\right)  . \label{3.2}%
\end{equation}

The assumption (\ref{1.4}) implies that $D\Delta$ is a strongly elliptic
operator in the sense of Petrowski, see Friedman \cite{Friedman}.

\begin{proposition}
\label{proposition1}Diagonalizing system (\ref{1.1})\ yields:%
\begin{equation}
W_{t}-\operatorname*{diag}\left(  \lambda_{1},\lambda_{2},...,\lambda
_{m}\right)  \Delta W=\digamma\left(  W\right)  \text{ in }\Omega\times\left(
0,+\infty\right)  , \label{5.1}%
\end{equation}
with the boundary condition%
\begin{equation}
\alpha W+\left(  1-\alpha\right)  \partial_{\eta}W=\Lambda\text{\ \ \ \ \ on
}\partial\Omega\times\left(  0,+\infty\right)  , \label{5.2}%
\end{equation}
or
\begin{equation}
\alpha W+\left(  1-\alpha\right)  \operatorname*{diag}\left(  \lambda
_{1},\lambda_{2},...,\lambda_{m}\right)  \partial_{\eta}W=\Lambda\text{\ on
}\partial\Omega\times\left(  0,+\infty\right)  , \label{5.02}%
\end{equation}
and the initial data%
\[
W\left(  x,0\right)  =W_{0}\text{ }\ \ \ \ \text{on }\Omega
\]

\end{proposition}

\begin{proof}
The eigenvectors of the diffusion matrix associated with the eigenvalues
$\lambda_{\ell}$ are defined as $V_{\ell}=\left(  v_{\ell1},v_{\ell
2},...,v_{\ell m}\right)  ^{T}$. Let us consider the diagonalizing matrix of
eigenvectors $P=\left(  V_{1}\shortmid V_{2}\shortmid...\shortmid
V_{m}\right)  $ and define the solution vector $U$ and the reaction terms
vector $F$. Pre-multiplying the system by $P^{T}$\ yields
\begin{align}
U_{t}-D\Delta U  &  =F\nonumber\\
P^{T}U_{t}-\Delta P^{T}DU  &  =P^{T}F\nonumber\\
P^{T}U_{t}-\Delta P^{T}D\left(  P^{T}\right)  ^{-1}P^{T}U  &  =P^{T}F.
\label{P}%
\end{align}
The term $P^{T}U$ can be simplified as follows
\begin{align}
P^{T}U  &  =\left(  V_{1}\shortmid V_{2}\shortmid...\shortmid V_{m}\right)
^{T}U\nonumber\\
&  =\left(  \left\langle V_{1},U\right\rangle ,\left\langle V_{2}%
,U\right\rangle ,...,\left\langle V_{m},U\right\rangle \right)  ^{T}%
\nonumber\\
&  =\left(  w_{1},w_{2},...,w_{m}\right)  ^{T}=W. \label{P1}%
\end{align}
Hence, $P^{T}U_{t}=W_{t}$. Similarly,%
\begin{align}
P^{T}F  &  =\left(  V_{1}\shortmid V_{2}\shortmid...\shortmid V_{m}\right)
^{T}F\nonumber\\
&  =\left(  \left\langle V_{1},F\right\rangle ,\left\langle V_{2}%
,F\right\rangle ,...,\left\langle V_{m},F\right\rangle \right)  ^{T}%
\nonumber\\
&  =\left(  \digamma_{1},\digamma_{2},...,\digamma_{m}\right)  ^{T}=\digamma.
\label{P2}%
\end{align}
Furthermore we have the similarity transformation
\begin{align}
P^{T}D\left(  P^{T}\right)  ^{-1}  &  =P^{T}\left(  D^{T}\right)  ^{T}\left(
P^{-1}\right)  ^{T}\nonumber\\
&  =\left(  D^{T}P\right)  ^{T}\left(  P^{-1}\right)  ^{T}\nonumber\\
&  =\left(  P^{-1}D^{T}P\right)  ^{T}\nonumber\\
&  =\left(  \operatorname*{diag}\left(  \lambda_{1},\lambda_{2},...,\lambda
_{m}\right)  \right)  ^{T}\nonumber\\
&  =\operatorname*{diag}\left(  \lambda_{1},\lambda_{2},...,\lambda
_{m}\right)  . \label{P3}%
\end{align}

Substituting (\ref{P1}), (\ref{P2}), and (\ref{P3}) in (\ref{P}) results in
the equivalent system (\ref{5.1}). The boundary condition (\ref{5.2}) can be
obtained by pre-multiplying (\ref{1.2}) by $P^{T}$:%
\begin{align}
\alpha U+\left(  1-\alpha\right)  \partial_{\eta}U  &  =B,\nonumber\\
\alpha P^{T}U+\left(  1-\alpha\right)  \partial_{\eta}P^{T}U  &  =P^{T}B.
\label{P4}%
\end{align}
Simplifying the term $P^{T}B$ yields%
\begin{align}
P^{T}B  &  =\left(  V_{1}\shortmid V_{2}\shortmid...\shortmid V_{m}\right)
^{T}B\nonumber\\
&  =\left(  \left\langle V_{1},B\right\rangle ,\left\langle V_{2}%
,B\right\rangle ,...,\left\langle V_{m},B\right\rangle \right)  ^{T}%
\nonumber\\
&  =\left(  \rho_{1}^{0},\rho_{2}^{0},...,\rho_{m}^{0}\right)  ^{T}:=\Lambda.
\label{P5}%
\end{align}

Substituting (\ref{P1}) and (\ref{P4}) in (\ref{P5}) gives the boundary
condition for the equivalent system (\ref{5.2}).

Pre-multiplying (\ref{1.02}) by $P^{T}$\ yields%
\begin{align}
\alpha U+\left(  1-\alpha\right)  D\partial_{\eta}U  &  =B,\nonumber\\
\alpha P^{T}U+\left(  1-\alpha\right)  \partial_{\eta}P^{T}DU  &
=P^{T}B,\nonumber\\
\alpha P^{T}U+\left(  1-\alpha\right)  \partial_{\eta}P^{T}D\left(
P^{T}\right)  ^{-1}P^{T}U  &  =P^{T}B. \label{P6}%
\end{align}

Substituting (\ref{P2}), (\ref{P3}), and (\ref{P5}) results in the equivalent
boundary condition in (\ref{5.02}).

We note that condition (\ref{1.4}) guarantees the parabolicity of the system
(\ref{1.1}), which implies that this system is equivalent to that described by
(\ref{5.1}) in the region:
\begin{align*}
\Sigma_{\mathfrak{L},\mathfrak{Z}}  &  =\left\{  U_{0}\in\mathbb{%
\mathbb{R}
}^{m}:\left\langle V_{\ell},U_{0}\right\rangle \geq0,\text{ }\ell
\in\mathfrak{L}\right\} \\
&  =\left\{  U_{0}\in\mathbb{%
\mathbb{R}
}^{m}:w_{\ell}^{0}=\left\langle V_{\ell},U_{0}\right\rangle \geq0,\text{ }%
\ell\in\mathfrak{L}\right\}
\end{align*}
with%
\[
\rho_{\ell}^{0}=\left\langle V_{\ell},B\right\rangle \geq0,\ell\in
\mathfrak{L}.
\]
This implies that the components $w_{\ell}$ are necessarily positive.
\end{proof}

The local existence and uniqueness of solutions to the initial system
(\ref{1.1}), with initial data in $C(\overline{\Omega})$ or $L^{p}(\Omega)$,
$p\in\left(  1,+\infty\right)  $, follows from the basic existence theory for
abstract semi-linear differential equations (Henry \cite{Henry} ). The
solutions are classical on $\left(  0,T_{\max}\right)  $, where $T_{\max}$
denotes the eventual blow up time in $L^{\infty}(\Omega)$. The local solution
is continued globally by apriori estimates. Once the invariant regions are
constructed, one can apply the Lyapunov technique and establish the global
existence of a unique solution for (\ref{1.1}).

\begin{proposition}
The system (\ref{5.1}) admits a unique classical solution $W$ on$\ \Omega
\times(0,T_{\max})$; moreover we have the alternative%
\begin{equation}
\text{If }T_{\max}<\infty\text{ then }\underset{t\nearrow T_{\max}}{\lim
}\overset{m}{\underset{\ell=1}{\sum}}\left\Vert w_{\ell}\left(  t,.\right)
\right\Vert _{\infty}=\infty\text{,} \label{3.3}%
\end{equation}
where $T_{\max}$ $\left(  \left\Vert w_{1}^{0}\right\Vert _{\infty},\left\Vert
w_{2}^{0}\right\Vert _{\infty},...,\left\Vert w_{m}^{0}\right\Vert _{\infty
}\right)  $ denotes the eventual blow-up time.
\end{proposition}

\section{Main Result}

Before we present the main result of this paper, let us define%

\begin{equation}
K_{l}^{r}=K_{r-1}^{r-1}K_{l}^{r-1}-\left[  H_{l}^{r-1}\right]  ^{2},\text{
}r=3,...,l\text{,} \label{mycond}%
\end{equation}
where%
\[
H_{l}^{r}=\underset{1\leq\ell,\kappa\leq l}{\det}\left(  \left(
a_{\ell,\kappa}\right)  _{\substack{\ell\neq l,...r+1\\\kappa\neq
l-1,..r}}\right)  \overset{k=r-2}{\underset{k=1}{\prod}}\left(  \det\left[
k\right]  \right)  ^{2^{\left(  r-k-2\right)  }},\text{ }r=3,...,l-1\text{,}%
\]%
\[
K_{l}^{2}=\underset{\text{positive value}}{\underbrace{\bar{\lambda}_{1}%
\bar{\lambda}_{l}\overset{l-1}{\underset{k=1}{\prod}}\theta_{k}^{2\left(
p_{k}+1\right)  ^{2}}\overset{m-1}{\underset{k=l}{\prod}}\theta_{k}^{2\left(
p_{k}+2\right)  ^{2}}}}\left[  \overset{l-1}{\underset{k=1}{\prod}}\theta
_{k}^{2}-A_{1l}^{2}\right]  ,
\]
and%
\[
H_{l}^{2}=\underset{\text{positive value}}{\underbrace{\bar{\lambda}_{1}%
\sqrt{\bar{\lambda}_{2}\bar{\lambda}_{l}}\theta_{1}^{2\left(  p_{1}+1\right)
^{2}}\overset{l-1}{\underset{k=2}{\prod}}\theta_{k}^{\left(  p_{k}+2\right)
^{2}+\left(  p_{k}+1\right)  ^{2}}\overset{m-1}{\underset{k=l}{\prod}}%
\theta_{k}^{2\left(  p_{k}+2\right)  ^{2}}}}\left[  \theta_{1}^{2}%
A_{2l}-A_{12}A_{1l}\right]  .
\]
Here $\underset{1\leq\ell,\kappa\leq l}{\det}\left(  \left(  a_{\ell,\kappa
}\right)  _{\substack{\ell\neq l,...r+1\\\kappa\neq l-1,..r}}\right)  $
denotes the determinant of the $r$-square symmetric matrix obtained from
$\left(  a_{\ell,\kappa}\right)  _{1\leq\ell,\kappa\leq m}$ by removing the
$\left(  r+1\right)  ^{\text{th}}$, $\left(  r+2\right)  ^{\text{th}}$, ...,
$l^{\text{th}}$ rows and the $r^{\text{th}},\left(  r+1\right)  ^{\text{th}%
},...,\left(  l-1\right)  ^{\text{th}}$ columns, and $\det\left[  1\right]
,...,\det\left[  m\right]  $ \ are the minors of the matrix $\left(
a_{\ell,\kappa}\right)  _{1\leq\ell,\kappa\leq m}.$ The elements of the matrix
are:%
\begin{equation}
a_{\ell\kappa}=\frac{\lambda_{\ell}+\lambda_{\kappa}}{2}\theta_{1}^{p_{1}^{2}%
}...\theta_{\left(  \ell-1\right)  }^{p_{\left(  \ell-1\right)  }^{2}}%
\theta_{\ell}^{\left(  p_{\ell}+1\right)  ^{2}}...\theta_{\kappa-1}^{\left(
p_{\left(  \kappa-1\right)  }+1\right)  ^{2}}\theta_{\kappa}^{\left(
p_{\kappa}+2\right)  ^{2}}...\theta_{\left(  m-1\right)  }^{\left(  p_{\left(
m-1\right)  }+2\right)  ^{2}}, \label{1.13}%
\end{equation}
where $\lambda_{\ell}$ is defined in (\ref{2.2})-(\ref{2.3}). Note that
$A_{\ell\kappa}=\dfrac{\lambda_{\ell}+\lambda_{\kappa}}{2\sqrt{\lambda_{\ell
}\lambda_{\kappa}}}$ for all $\ell,\kappa=1,...,m$, and $\theta_{\ell},$
$\ell=1,...,\left(  m-1\right)  $ are positive constants.

\begin{theorem}
\label{theorem1} Suppose that the functions $\digamma_{\ell},$ $\ell=1,...,m$,
are of polynomial growth and satisfy the condition (\ref{1.11}) for some
sufficiently large positive constants $D_{\ell},$ $\ell=1,...,m$. Let $\left(
w_{1}\left(  t,.\right)  ,w_{2}\left(  t,.\right)  ,...,w_{m}\left(
t,.\right)  \right)  $ be a solution of (\ref{5.1})-(\ref{5.2}) and%
\begin{equation}
L(t)=\int_{\Omega}H_{p_{m}}\left(  w_{1}\left(  t,x\right)  ,w_{2}\left(
t,x\right)  ,...,w_{m}\left(  t,x\right)  \right)  dx, \label{5.3}%
\end{equation}
where%
\[
H_{p_{m}}\left(  w_{1},...,w_{m}\right)  =\overset{p_{m}}{\underset{p_{m-1}%
=0}{\sum}}...\overset{p_{2}}{\underset{p_{1}=0}{\sum}}C_{p_{m}}^{p_{m-1}%
}...C_{p_{2}}^{p_{1}}\theta_{1}^{p_{1}^{2}}...\theta_{\left(  m-1\right)
}^{p_{\left(  m-1\right)  }^{2}}w_{1}^{p_{1}}w_{2}^{p_{2}-p_{1}}%
...w_{m}^{p_{m}-p_{m-1}},
\]
with $p_{m}$ a positive integer and $C_{p_{\kappa}}^{p_{\ell}}=\frac
{p_{\kappa}!}{p_{\ell}!\left(  p_{\kappa}-p_{\ell}\right)  !}$.\newline
Furthermore suppose that the following condition is satisfied%
\begin{equation}
K_{l}^{l}>0,\text{ }l=2,...,m\text{.} \label{1.12}%
\end{equation}
where $K_{l}^{l}$ was defined in (\ref{mycond}).Then it follows that the
functional $L$ is uniformly bounded on the interval $\left[  0,T^{\ast
}\right]  ,$ $T^{\ast}<T_{\max}$.
\end{theorem}

\begin{corollary}
\label{corollary1}Under the assumptions of Theorem \ref{theorem1}, all
solutions of (\ref{5.1})-(\ref{5.2}) with positive initial data in $L^{\infty
}\left(  \Omega\right)  $ are in $L^{\infty}\left(  0,T^{\ast};L^{p}\left(
\Omega\right)  \right)  $, for some $p\geq1.$
\end{corollary}

\begin{proposition}
\label{proposition2}Under the assumptions of theorem \ref{theorem1} and given
that the condition (\ref{1.4}) is satisfied, all solutions of (\ref{5.1}%
)-(\ref{5.2}) with positive initial data in $L^{\infty}\left(  \Omega\right)
$ are global for some $p>\dfrac{Nn}{2}$.
\end{proposition}

For the proof of Theorem \ref{theorem1}, we first need to define some
preparatory Lemmas.

\begin{lemma}
[see \cite{Abdelmalek3}]\label{Lemma1} let $H_{p_{m}}$ be the homogeneous
polynomial defined in (\ref{5.3}), we have%
\begin{align}
\partial_{w_{1}}H_{p_{m}}  &  =p_{m}\overset{p_{m}-1}{\underset{p_{m-1}%
=0}{\sum}}...\overset{p_{2}}{\underset{p_{1}=0}{\sum}}C_{p_{m}-1}^{p_{m-1}%
}...C_{p_{2}}^{p_{1}}\theta_{1}^{\left(  p_{1}+1\right)  ^{2}}...\theta
_{\left(  m-1\right)  }^{\left(  p_{\left(  m-1\right)  }+1\right)  ^{2}%
}\nonumber\\
&  w_{1}^{p_{1}}w_{2}^{p_{2}-p_{1}}w_{3}^{p_{3}-p_{2}}...w_{m}^{\left(
p_{m}-1\right)  -p_{m-1}}. \label{5.4}%
\end{align}%
\begin{align}
\partial_{w_{\ell}}H_{p_{m}}  &  =p_{m}\overset{p_{m}-1}{\underset{p_{m-1}%
=0}{\sum}}...\overset{p_{2}}{\underset{p_{1}=0}{\sum}}C_{p_{m}-1}^{p_{m-1}%
}...C_{p_{2}}^{p_{1}}\theta_{1}^{p_{1}^{2}}...\theta_{\ell-1}^{p_{\left(
\ell-1\right)  }^{2}}\theta_{\ell}^{\left(  p_{\ell}+1\right)  2}%
...\theta_{\left(  m-1\right)  }^{\left(  p_{\left(  m-1\right)  }+1\right)
^{2}}\nonumber\\
&  w_{1}^{p_{1}}w_{2}^{p_{2}-p_{1}}w_{3}^{p_{3}-p_{2}}...w_{m}^{\left(
p_{m}-1\right)  -p_{m-1}},\text{ }\ell=2,...,m-1, \label{5.5}%
\end{align}%
\begin{align}
\partial_{w_{m}}H_{p_{m}}  &  =p_{m}\overset{p_{m}-1}{\underset{p_{m-1}%
=0}{\sum}}...\overset{p_{2}}{\underset{p_{1}=0}{\sum}}C_{p_{m}-1}^{p_{m-1}%
}...C_{p_{3}}^{p_{2}}C_{p_{2}}^{p_{1}}\theta_{1}^{p_{1}^{2}}\theta_{2}%
^{p_{2}^{2}}...\theta_{\left(  m-1\right)  }^{p_{\left(  m-1\right)  }^{2}%
}\nonumber\\
&  w_{1}^{p_{1}}w_{2}^{p_{2}-p_{1}}w_{3}^{p_{3}-p_{2}}...w_{m}^{\left(
p_{m}-1\right)  -p_{m-1}}. \label{5.6}%
\end{align}

\end{lemma}

\begin{lemma}
[see \cite{Abdelmalek3}]\label{Lemma2} We have%
\begin{align}
\partial_{w_{1}^{2}}H_{n}  &  =p_{m}\left(  p_{m}-1\right)  \overset{p_{m}%
-2}{\underset{p_{m-1}=0}{\sum}}...\overset{p_{3}}{\underset{p_{2}=0}{\sum}%
}\overset{p_{2}}{\underset{p_{1}=0}{\sum}}C_{p_{m}-2}^{p_{m-1}}...C_{p_{2}%
}^{p_{1}}\nonumber\\
&  \theta_{1}^{\left(  p_{1}+2\right)  ^{2}}...\theta_{\left(  m-1\right)
}^{\left(  p_{\left(  m-1\right)  }+2\right)  ^{2}}w_{1}^{p_{1}}w_{2}%
^{p_{2}-p_{1}}...w_{m}^{\left(  p_{m}-2\right)  -p_{m-1}}, \label{5.7}%
\end{align}%
\begin{align}
\partial_{w_{\ell}^{2}}H_{n}  &  =p_{m}\left(  p_{m}-1\right)  \overset
{p_{m}-2}{\underset{p_{m-1}=0}{\sum}}...\overset{p_{2}}{\underset{p_{1}%
=0}{\sum}}C_{p_{m}-2}^{p_{m-1}}...C_{p_{2}}^{p_{1}}\nonumber\\
&  \theta_{1}^{p_{1}^{2}}\theta_{2}^{p_{2}^{2}}...\theta_{\ell-1}^{p_{\ell
-1}^{2}}\theta_{\ell}^{\left(  p_{\ell}+2\right)  ^{2}}...\theta_{\left(
m-1\right)  }^{\left(  p_{\left(  m-1\right)  }+2\right)  ^{2}}\nonumber\\
&  w_{1}^{p_{1}}w_{2}^{p_{2}-p_{1}}...w_{m}^{\left(  p_{m}-2\right)  -p_{m-1}%
}, \label{5.8}%
\end{align}
for all\ $\ell=2,...,m-1$, and
\begin{align}
\partial_{w_{\ell}w_{\kappa}}H_{n}  &  =p_{m}\left(  p_{m}-1\right)
\overset{p_{m}-2}{\underset{p_{m-1}=0}{\sum}}...\overset{p_{2}}{\underset
{p_{1}=0}{\sum}}C_{p_{m}-2}^{p_{m-1}}...C_{p_{2}}^{p_{1}}\nonumber\\
&  \theta_{1}^{p_{1}^{2}}...\theta_{\ell-1}^{p_{\ell-1}^{2}}\theta_{\ell
}^{\left(  p_{\ell}+1\right)  ^{2}}...\theta_{\kappa-1}^{\left(  p_{\kappa
-1}+1\right)  ^{2}}\theta_{\kappa}^{\left(  p_{\kappa}+2\right)  ^{2}%
}...\theta_{\left(  m-1\right)  }^{\left(  p_{\left(  m-1\right)  }+2\right)
^{2}}\nonumber\\
&  w_{1}^{p_{1}}w_{2}^{p_{2}-p_{1}}...w_{m}^{\left(  p_{m}-2\right)  -p_{m-1}}
\label{5.9}%
\end{align}
for all\ $1\leq\ell<\kappa\leq m$,%
\begin{align}
\partial_{w_{m}^{2}}H_{n}  &  =p_{m}\left(  p_{m}-1\right)  \overset{p_{m}%
-2}{\underset{p_{m-1}=0}{\sum}}...\overset{p_{2}}{\underset{p_{1}=0}{\sum}%
}C_{p_{m}-2}^{p_{m-1}}...C_{p_{2}}^{p_{1}}\theta_{1}^{p_{1}^{2}}%
...\theta_{\left(  m-1\right)  }^{p_{\left(  m-1\right)  }^{2}}\nonumber\\
&  w_{1}^{p_{1}}w_{2}^{p_{2}-p_{1}}...w_{m}^{\left(  p_{m}-2\right)  -p_{m-1}%
}. \label{5.10}%
\end{align}

\end{lemma}

\begin{lemma}
[see \cite{Abdelmalek3}]\label{Lemma3}Let $A$ be the\ $m$-square symmetric
matrix defined by $A=\left(  a_{\ell\kappa}\right)  _{1\leq\ell,\kappa\leq m}%
$. Then the following property holds:%
\begin{equation}
\left\{
\begin{array}
[c]{l}%
K_{m}^{m}=\det\left[  m\right]  \overset{k=m-2}{\underset{k=1}{\prod}}\left(
\det\left[  k\right]  \right)  ^{2^{\left(  m-k-2\right)  }},\text{
\ \ }m>2,\\
K_{2}^{2}=\det\left[  2\right]  ,
\end{array}
\right.  \label{5.11}%
\end{equation}
where%
\begin{align*}
K_{m}^{l}  &  =K_{l-1}^{l-1}K_{m}^{l-1}-\left(  H_{m}^{l-1}\right)
^{2},l=3,...,m,\\
H_{m}^{l}  &  =\underset{1\leq\ell,\kappa\leq m}{\det}\left(  \left(
a_{\ell,\kappa}\right)  _{\substack{\ell\neq m,...l+1\\\kappa\neq
m-1,..l}}\right)  \overset{k=l-2}{\underset{k=1}{\prod}}\left(  \det\left[
k\right]  \right)  ^{2^{\left(  l-k-2\right)  }},\text{ }l=3,...,m-1,\\
K_{m}^{2}  &  =a_{11}a_{mm}-\left(  a_{1m}\right)  ^{2},\text{ }H_{m}%
^{2}=a_{11}a_{2m}-a_{12}a_{1m}.
\end{align*}

\end{lemma}

\begin{proof}
[Proof of Theorem \ref{theorem1}]we prove that $L(t)$ is uniformly bounded on
the interval $\left[  0,T^{\ast}\right]  ,T^{\ast}<T_{\max}$. we have:%
\begin{align*}
L^{\prime}(t)  &  =\int_{\Omega}\partial_{t}H_{p_{m}}dx=\int_{\Omega}%
\overset{m}{\underset{\ell=1}{\sum}}\partial_{w_{\ell}}H_{p_{m}}\frac{\partial
w_{\ell}}{\partial t}dx\\
&  =\int_{\Omega}\overset{m}{\underset{\ell=1}{\sum}}\partial_{w_{\ell}%
}H_{p_{m}}\left(  \lambda_{\ell}\Delta w_{\ell}+\digamma_{\ell}\right)  dx\\
&  =\int_{\Omega}\overset{m}{\underset{\ell=1}{\sum}}\lambda_{\ell}%
\partial_{w_{\ell}}H_{p_{m}}\Delta w_{\ell}dx+\int_{\Omega}\overset
{m}{\underset{\ell=1}{\sum}}\partial_{w_{\ell}}H_{p_{m}}\digamma_{\ell}dx=I+J,
\end{align*}
where%
\begin{equation}
I=\int_{\Omega}\overset{m}{\underset{\ell=1}{\sum}}\lambda_{\ell}%
\partial_{w_{\ell}}H_{p_{m}}\Delta w_{\ell}dx, \label{5.11a}%
\end{equation}
and%
\begin{equation}
J=\int_{\Omega}\overset{m}{\underset{\ell=1}{%
{\displaystyle\sum}
}}\partial_{w_{\ell}}H_{p_{m}}\digamma_{\ell}dx. \label{5.12}%
\end{equation}

Using Green's formula we can divide $I$ into two parts: $I_{1}$ and $I_{2}$,
where
\begin{equation}
I_{1}=\int_{\partial\Omega}\overset{m}{\underset{\ell=1}{\sum}}\lambda_{\ell
}\partial_{w_{\ell}}H_{p_{m}}\partial_{\eta}w_{\ell}dx, \label{5.13}%
\end{equation}
and
\begin{equation}
I_{2}=-\int_{\Omega}\left\langle T,\left(  \left(  \frac{\lambda_{\ell
}+\lambda_{\kappa}}{2}\partial_{w_{\kappa}w_{\ell}}H_{p_{m}}\right)
_{1\leq\ell,\kappa\leq m}\right)  T\right\rangle dx, \label{5.14}%
\end{equation}
for $p_{1}=0,...,p_{2},$ $p_{2}=0,...,p_{3}$ $...p_{m-1}=0,...,p_{m}-2$ and
\newline$T=\left(  \nabla w_{1},\nabla w_{2},...,\nabla w_{m}\right)  ^{T}.$
Applying Lemmas \ref{Lemma1} and \ref{Lemma2} yields%
\begin{equation}
\left.
\begin{array}
[c]{l}%
\left(  \frac{\lambda_{\ell}+\lambda_{\kappa}}{2}\partial_{w_{\kappa}w_{\ell}%
}H_{p_{m}}\right)  _{1\leq\ell,\kappa\leq m}=\\
p_{m}\left(  p_{m}-1\right)  \overset{p_{m}-2}{\underset{p_{m-1}=0}{\sum}%
}...\overset{p_{2}}{\underset{p_{1}=0}{\sum}}C_{p_{m}-2}^{p_{m-1}}...C_{p_{2}%
}^{p_{1}}\left(  \left(  a_{\ell\kappa}\right)  _{1\leq\ell,\kappa\leq
m}\right)  w_{1}^{p_{1}}...w_{m}^{\left(  p_{m}-2\right)  -p_{m-1}},
\end{array}
\right.  \label{5.15}%
\end{equation}
where $\left(  a_{\ell\kappa}\right)  _{1\leq\ell,\kappa\leq m}$ is the matrix
defined in (\ref{1.13}).

Now, in order to prove that $I$ is bounded, we will show that there exists a
positive constant $C_{4}$ independent of $t\in\left[  0,T_{\max}\right)  $
such that
\begin{equation}
I_{1}\leq C_{4}\text{ for all }t\in\left[  0,T_{\max}\right)  , \label{5.16}%
\end{equation}
and that%
\begin{equation}
I_{2}\leq0, \label{5.17}%
\end{equation}
for several boundary conditions. First let us prove (\ref{5.16}):

(i) If $0<\alpha<1$ , then using the boundary conditions (\ref{1.2}) we get
\[
I_{1}=\int_{\partial\Omega}\overset{m}{\underset{\ell=1}{\sum}}\lambda_{\ell
}\partial_{w_{\ell}}H_{p_{m}}\left(  \gamma_{\ell}-\sigma_{\ell}w_{\ell
}\right)  dx,
\]
where $\sigma_{\ell}=\dfrac{\alpha}{1-\alpha}$ and $\gamma_{\ell}=\dfrac
{\beta_{\ell}}{1-\alpha_{\ell}}$, for\ $\ell=1,...m$. For the second type of
boundary condition (\ref{5.02}), $\sigma_{\ell}=\dfrac{\alpha}{\lambda_{\ell
}\left(  1-\alpha\right)  }$ and $\gamma_{\ell}=\dfrac{\beta_{\ell}}%
{\lambda_{\ell}\left(  1-\alpha\right)  }$.

Since $H\left(  W\right)  =\overset{m}{\underset{\ell=1}{\sum}}\lambda_{\ell
}\partial_{w_{\ell}}H_{p_{m}}\left(  \gamma_{\ell}-\sigma_{\ell}w_{\ell
}\right)  =P_{n-1}\left(  W\right)  -Q_{n}\left(  W\right)  $, where $P_{n-1}$
and $Q_{n}$ are polynomials with positive coefficients and respective degrees
$n-1$ and $n$, and since the solution is positive it follows that
\begin{equation}
\underset{\overset{m}{\underset{\ell=1}{\sum}}\left\vert w_{\ell}\right\vert
\rightarrow+\infty}{\lim\sup}H\left(  W\right)  =-\infty, \label{5.18}%
\end{equation}
which proves that $H$ is uniformly bounded on $%
\mathbb{R}
_{+}^{m}$\ and consequently proves (\ref{5.16}).

(ii)$\;$If for all $\ell=1,...m:\alpha=0$, then $I_{1}=0$ on $\left[
0,T_{\max}\right)  $.

(iii) The case of homogeneous Dirichlet conditions is trivial since in this
case the positivity of the solution on $\left[  0,T_{\max}\right)
\times\Omega$ implies $\partial_{\eta}w_{\ell}\leq0,\forall\ell=1,...m$ on
$\left[  0,T_{\max}\right)  \times\partial\Omega$. Consequently one obtains
the same result in (\ref{5.16}) with $C_{4}=0$.

Hence the proof of (\ref{5.16}) is complete.

Now we move to the proof of (\ref{5.17}).

Consider the matrix $\left(  a_{\ell\kappa}\right)  _{1\leq\ell,\kappa\leq m}$
which we defined in (\ref{1.13}). The quadratic form (with respect to $\nabla
w_{\ell},$ $\ell=1,...,m$) associated with the matrix $\left(  a_{\ell\kappa
}\right)  _{1\leq\ell,\kappa\leq m}$, with $p_{1}=0,...,p_{2},$ $p_{2}%
=0,..,p_{3}$ ... $p_{m-1}=0,...,p_{m}-2$, is positive definite since its
minors $\det\left[  1\right]  $, $\det\left[  2\right]  $,$...$ $\det\left[
m\right]  $ are all positive. Let us prove their positivity by induction.

The first minor
\[
\det\left[  1\right]  =\lambda_{1}\theta_{1}^{\left(  p_{1}+2\right)  ^{2}%
}\theta_{2}^{\left(  p_{2}+2\right)  ^{2}}..\theta_{\left(  m-1\right)
}^{\left(  p_{\left(  m-1\right)  }+2\right)  ^{2}}>0
\]
for $p_{1}=0,...,p_{2},$ $p_{2}=0,...,p_{3}$ ... $p_{m-1}=0,...,p_{m}-2$.

For the second minor $\det\left[  2\right]  $, and according to Lemma
\ref{Lemma3}, we have:
\[
\det\left[  2\right]  =K_{2}^{2}=\lambda_{1}\lambda_{2}\theta_{1}^{2\left(
p_{1}+1\right)  ^{2}}\overset{m-1}{\underset{k=2}{\Pi}}\theta_{k}^{2\left(
p_{k}+2\right)  ^{2}}\left[  \theta_{1}^{2}-A_{12}^{2}\right]  ,
\]
using (\ref{1.12})\ for $l=2$\ we get $\det\left[  2\right]  >0.$

Similarly for the third minor $\det\left[  3\right]  $, and again using Lemma
\ref{Lemma3}, we have:
\[
K_{3}^{3}=\det\left[  3\right]  \det\left[  1\right]  .
\]
Since $\det\left[  1\right]  >0$, we conclude that%
\[
\operatorname*{sign}(K_{3}^{3})=\operatorname*{sign}(\det\left[  3\right]  ).
\]
Again, using (\ref{1.12})\ for $l=3$ yields $\det\left[  3\right]  >0$.

To finish the proof let us suppose $\det\left[  k\right]  >0$ for
$k=1,2,...,l-1$ and show that $\det[l]$ is necessarily positive. We have
\begin{equation}
\det\left[  k\right]  >0,k=1,...,\left(  l-1\right)  \Rightarrow
\overset{k=l-2}{\underset{k=1}{\prod}}\left(  \det\left[  k\right]  \right)
^{2^{\left(  l-k-2\right)  }}>0. \label{5.20}%
\end{equation}
From Lemma \ref{Lemma3} we obtain $K_{l}^{l}=\det\left[  l\right]
\overset{k=l-2}{\underset{k=1}{\prod}}\left(  \det\left[  k\right]  \right)
^{2^{\left(  l-k-2\right)  }}$, and from (\ref{5.20}) we get
$\operatorname*{sign}(K_{l}^{l})=\operatorname*{sign}\left(  \det\left[
l\right]  \right)  $. Since $K_{l}^{l}>0$ according to (\ref{1.12}) then
$\det\left[  l\right]  >0$ and the proof of (\ref{5.17}) is concluded. It then
follows from (\ref{5.16}) and (\ref{5.17}) that $I$ is finished.

Now let us prove that $J$ in (\ref{5.12}) is bounded. Substituting the
expressions of the partial derivatives given by Lemma \ref{Lemma1} in the
second integral of (\ref{5.12}) yields%
\begin{align*}
J  &  =\int_{\Omega}\left[  p_{m}\overset{p_{m}-1}{\underset{p_{m-1}=0}{\sum}%
}...\overset{p_{2}}{\underset{p_{1}=0}{\sum}}C_{p_{m}-1}^{p_{m-1}}...C_{p_{2}%
}^{p_{1}}w_{1}^{p_{1}}w_{2}^{p_{2}-p_{1}}...w_{m}^{p_{m}-1-p_{m-1}}\right] \\
&  \left(  \overset{m-1}{\underset{\ell=1}{\Pi}}\theta_{\ell}^{\left(
p_{\ell}+1\right)  ^{2}}\digamma_{1}+\overset{m-1}{\underset{\kappa=2}{\sum}%
}\overset{\kappa-1}{\underset{k=1}{\prod}}\theta_{k}^{p_{k}^{2}}\overset
{m-1}{\underset{\ell=\kappa}{\prod}}\theta_{\ell}^{\left(  p_{\ell}+1\right)
^{2}}\digamma_{\kappa}+\overset{m-1}{\underset{\ell=1}{\prod}}\theta_{\ell
}^{p_{\ell}^{2}}\digamma_{m}\right)  dx\\
&  =\int_{\Omega}\left[  p_{m}\overset{p_{m}-1}{\underset{p_{m-1}=0}{\sum}%
}...\overset{p_{2}}{\underset{p_{1}=0}{\sum}}C_{p_{m}-1}^{p_{m-1}}...C_{p_{2}%
}^{p_{1}}w_{1}^{p_{1}}w_{2}^{p_{2}-p_{1}}...w_{m}^{p_{m}-1-p_{m-1}}\right] \\
&  \left(  \overset{m-1}{\underset{\ell=1}{\Pi}}\frac{\theta_{\ell}^{\left(
p_{\ell}+1\right)  ^{2}}}{\theta_{\ell}^{p_{\ell}^{2}}}\digamma_{1}%
+\overset{m-1}{\underset{\kappa=2}{\sum}}\overset{\kappa-1}{\underset
{k=1}{\prod}}\theta_{k}^{p_{k}^{2}}\overset{m-1}{\underset{\ell=\kappa}{\prod
}}\frac{\theta_{\ell}^{\left(  p_{\ell}+1\right)  ^{2}}}{\theta_{\ell
}^{p_{\ell}^{2}}}\digamma_{\kappa}+\digamma_{m}\right)  \overset
{m-1}{\underset{\ell=1}{\prod}}\theta_{\ell}^{p_{\ell}^{2}}dx\\
&  =\int_{\Omega}\left[  p_{m}\overset{p_{m}-1}{\underset{p_{m-1}=0}{\sum}%
}...\overset{p_{2}}{\underset{p_{1}=0}{\sum}}C_{p_{m}-1}^{p_{m-1}}...C_{p_{2}%
}^{p_{1}}w_{1}^{p_{1}}w_{2}^{p_{2}-p_{1}}...w_{m}^{p_{m}-1-p_{m-1}}\right] \\
&  \left\langle \left(  \overset{m-1}{\underset{\ell=1}{\prod}}\frac
{\theta_{\ell}^{\left(  p_{\ell}+1\right)  ^{2}}}{\theta_{\ell}^{p_{\ell}^{2}%
}},\theta_{1}^{p_{1}^{2}}\overset{m-1}{\underset{\ell=2}{\prod}}\frac
{\theta_{\ell}^{\left(  p_{\ell}+1\right)  ^{2}}}{\theta_{\ell}^{p_{\ell}^{2}%
}},....,\overset{m-2}{\underset{k=1}{\prod}}\theta_{k}^{p_{k}^{2}}\frac
{\theta_{m-1}^{\left(  p_{m-1}+1\right)  ^{2}}}{\theta_{m-1}^{p_{m-1}^{2}}%
},1\right)  ,\digamma\right\rangle \overset{m-1}{\underset{\ell=1}{\prod}%
}\theta_{\ell}^{p_{\ell}^{2}}dx.
\end{align*}
Hence using the condition (\ref{1.11}) we deduce that
\[
J\leq C_{5}\int_{\Omega}\left[  \overset{p_{m}-1}{\underset{p_{m-1}=0}{\sum}%
}...\overset{p_{2}}{\underset{p_{1}=0}{\sum}}C_{p_{2}}^{p_{1}}...C_{p_{m}%
-1}^{p_{m-1}}w_{1}^{p_{1}}w_{2}^{p_{2}-p_{1}}...w_{m}^{p_{m}-1-p_{m-1}}\left(
1+\left\langle W,1\right\rangle \right)  \right]  dx.
\]
To prove that the functional $L$ is uniformly bounded on the interval $\left[
0,T^{\ast}\right]  $ we write
\begin{align*}
&  \overset{p_{m}-1}{\underset{p_{m-1}=0}{\sum}}...\overset{p_{2}}%
{\underset{p_{1}=0}{\sum}}C_{p_{2}}^{p_{1}}...C_{p_{m}-1}^{p_{m-1}}%
w_{1}^{p_{1}}w_{2}^{p_{2}-p_{1}}...w_{m}^{p_{m}-1-p_{m-1}}\left(
1+\left\langle W,1\right\rangle \right) \\
&  =R_{p_{m}}\left(  W\right)  +S_{p_{m}-1}\left(  W\right)  ,
\end{align*}
where $R_{p_{m}}\left(  W\right)  $ and $S_{p_{m}-1}\left(  W\right)  $ are
two homogeneous polynomials of degrees $p_{m}$ and $p_{m}-1,$ respectively.
Since all the polynomials $H_{p_{m}}$ and $R_{p_{m}}$ are of degree $p_{m}$
then there exists a positive constant $C_{6}$ such that
\begin{equation}
\int_{\Omega}R_{p_{m}}\left(  W\right)  dx\leq C_{6}\int_{\Omega}H_{p_{m}%
}\left(  W\right)  dx. \label{5.21}%
\end{equation}
Applying H\"{o}lder's inequality to the integral $\int_{\Omega}S_{p_{m}%
-1}\left(  W\right)  dx,$ one obtains
\[
\int_{\Omega}S_{p_{m}-1}\left(  W\right)  dx\leq\left(  meas\Omega\right)
^{\frac{1}{p_{m}}}\left(  \int_{\Omega}\left(  S_{p_{m}-1}\left(  W\right)
\right)  ^{\frac{p_{m}}{p_{m}-1}}dx\right)  ^{\frac{p_{m}-1}{p_{m}}}.
\]

Using the fact that for all $w_{1},w_{2,...},w_{m-1}\geq0$ and $w_{m}>0,$
\[
\dfrac{\left(  S_{p_{m}-1}\left(  W\right)  \right)  ^{\frac{p_{m}}{p_{m}-1}}%
}{H_{p_{m}}\left(  W\right)  }=\dfrac{\left(  S_{p_{m}-1}\left(  x_{1}%
,x_{2},...,x_{m-1},1\right)  \right)  ^{\frac{p_{m}}{p_{m}-1}}}{H_{p_{m}%
}\left(  x_{1},x_{2},...,x_{m-1},1\right)  },
\]
where we have $\forall\ell\in\left\{  1,2,...,m-1\right\}  :x_{\ell}%
=\frac{w_{\ell}}{w_{\ell+1}}$, and
\[
\underset{x_{\ell}\rightarrow+\infty}{\lim}\dfrac{\left(  S_{p_{m}-1}\left(
x_{1},x_{2},...,x_{m-1},1\right)  \right)  ^{\frac{p_{m}}{p_{m}-1}}}{H_{p_{m}%
}\left(  x_{1},x_{2},...,x_{m-1},1\right)  }<+\infty,
\]
one asserts that there exists a positive constant $C_{7}$ such that
\begin{equation}
\dfrac{\left(  S_{p_{m}-1}\left(  W\right)  \right)  ^{\frac{p_{m}}{p_{m}-1}}%
}{H_{p_{m}}\left(  W\right)  }\leq C_{7},\text{ for all }w_{1},w_{2}%
,...,w_{m}\geq0. \label{5.22}%
\end{equation}

Hence the functional $L$ satisfies the differential inequality
\[
L^{\prime}\left(  t\right)  \leq C_{6}L\left(  t\right)  +C_{8}L^{\frac
{p_{m}-1}{p_{m}}}\left(  t\right)  ,
\]
which for $Z=L^{\frac{1}{p_{m}}}$ can be written as
\begin{equation}
p_{m}Z^{\prime}\leq C_{6}Z+C_{8}. \label{5.23}%
\end{equation}
A simple integration gives the uniform bound of the functional $L$ on the
interval $\left[  0,T^{\ast}\right]  $. This ends the proof of the theorem.
\end{proof}

\begin{proof}
[Proof of Corollary \ref{corollary1}]It is an immediate consequence of Theorem
\ref{theorem1} and the inequality%
\begin{equation}
\int_{\Omega}\left\langle W,1\right\rangle ^{p}dx\leq C_{9}L\left(  t\right)
\text{ on }\left[  0,T^{\ast}\right]  . \label{5.24}%
\end{equation}
for some $p\geq1.$
\end{proof}

\begin{proof}
[Proof of Proposition \ref{proposition2}]From Corollary \ref{corollary1}, it
follows that there exists a positive constant $C_{10}$ such that
\begin{equation}
\int_{\Omega}\left(  \left\langle W,1\right\rangle +1\right)  ^{p}dx\leq
C_{10}\text{ on }\left[  0,T_{\max}\right)  . \label{5.25}%
\end{equation}
From (\ref{1.10}),\ we have%
\begin{align}
\text{for any }\ell &  \in\left\{  1,2,...,m\right\}  :\nonumber\\
\left\vert \digamma_{\ell}\left(  W\right)  \right\vert ^{\frac{p}{N}}  &
\leq C_{11}\left(  W\right)  \left\langle W,1\right\rangle ^{p}\text{ on
}\left[  0,T_{\max}\right)  \times\Omega. \label{5.26}%
\end{align}
Since $w_{1},w_{2},...,w_{m}$ are in $L^{\infty}\left(  0,T^{\ast}%
;L^{p}\left(  \Omega\right)  \right)  $ and $\dfrac{p}{N}>\dfrac{n}{2},$ then
as discussed in section \ref{sec:prelim}, the solution is global.
\end{proof}

\section{Construction of Invariant Regions}

Recall that the eigenvector of the diffusion matrix associated with the
eigenvalue $\lambda_{\ell}$ is defined as $V_{\ell}=\left(  v_{\ell1}%
,v_{\ell2},...,v_{\ell m}\right)  ^{T}$. In the region that we considered in
previous sections, we used the diagonalizing matrix $P=\left(  V_{1}\shortmid
V_{2}\shortmid...\shortmid V_{m}\right)  $. In general the diagonalizing
matrix can be written as%
\[
P=\left(  \left(  -1\right)  ^{i_{1}}V_{1}\shortmid\left(  -1\right)  ^{i_{2}%
}V_{2}\shortmid...\shortmid\left(  -1\right)  ^{i_{m}}V_{m}\right)  ,
\]
with the powers $i_{\ell}$
\[
i_{\ell}=1\text{ or }2,\text{ for }\ell=1,...,m.
\]
Now one can subdivide the indices $\ell$ into two disjoint sets $\mathfrak{Z}$
and $\mathfrak{L}$, such that
\[
\left\{
\begin{array}
[c]{c}%
i_{\ell}=1\Rightarrow\ell\in\mathfrak{Z}\\
i_{\ell}=2\Rightarrow\ell\in\mathfrak{L.}%
\end{array}
\right.
\]
It is then straightforward to notice that%
\[
\mathfrak{L}\cap\mathfrak{Z}=\phi,\text{ }\mathfrak{L}\cup\mathfrak{Z}%
=\left\{  1,2,...,m\right\}  .
\]
Hence the number of possible permutations for $\mathfrak{Z}$ and
$\mathfrak{L}$ is $2^{m}$.

Recall that
\[
W_{0}=P^{T}U_{0}=\left(  w_{1}^{0},w_{2}^{0},...,w_{m}^{0}\right)  ^{T}.
\]
Since we have $2^{m}$ different diagonalizing matrices $P^{T},$ we can write%
\[
W_{0}=\left\{
\begin{array}
[c]{l}%
w_{\ell}^{0}=\left\langle V_{\ell},U_{0}\right\rangle ,\text{ }\ell
\in\mathfrak{L,}\\
w_{\ell}^{0}=\left\langle \left(  -1\right)  V_{\ell},U_{0}\right\rangle
,\text{ }\ell\in\mathfrak{Z.}%
\end{array}
\right.
\]
This along with (\ref{1.6}) guarantees that the elements of $W_{0}$ are
positive, i.e.%
\[
\left.
\begin{array}
[c]{l}%
\Sigma_{\mathfrak{L},\mathfrak{Z}}=\left\{  U_{0}\in\mathbb{%
\mathbb{R}
}^{m}:w_{\ell}^{0}=\left\langle V_{\ell},U_{0}\right\rangle \geq0,\text{
}\right. \\
\text{ \ \ \ \ \ \ \ \ \ \ \ \ \ \ }\left.  \ell\in\mathfrak{L,}w_{\ell}%
^{0}=\left\langle \left(  -1\right)  V_{\ell},U_{0}\right\rangle \geq0,\text{
}\ell\in\mathfrak{Z}\right\}  .
\end{array}
\right.
\]

\end{document}